\documentclass[a4paper, 12pt] {article}
\usepackage[cp1251]{inputenc}
\usepackage[T1]{fontenc}
\usepackage[russian, english]{babel}
\usepackage{amsfonts}
\usepackage{mathtext}
\usepackage{cite}

\usepackage{graphicx}

\usepackage{verbatim}
\usepackage{amsmath}
\usepackage{amsthm}
\usepackage{amssymb}
\usepackage{delarray}
\usepackage{mathrsfs}
\usepackage{xcolor}

\theoremstyle{plain}
\newtheorem{Theorem}{\indent Theorem}
\newtheorem{Lemma}{\indent Lemma}

\newtheorem{Proposition}{\indent Proposition}

\theoremstyle{remark}

\theoremstyle{plain}
\newtheorem*{Theorem*}{\indent Theorem}
\newtheorem*{Lemma*}{\indent Lemma}
\newtheorem*{Corollary*}{\indent Corollary}
\newtheorem*{Hypothesis*}{\indent Hypothesis}
\newtheorem*{Proposition*}{\indent Proposition}
\theoremstyle{definition}
\newtheorem*{Remark*}{\indent Remark}
\newtheorem*{Question*}{\indent Question}
\newtheorem*{Definition*}{\indent Definition}
\newtheorem*{Notation*}{\indent Notation}

\newenvironment{Proof} {\indent\textbf{Proof. }}{\hfill$\Box$}

\begin{document}
\begin{center}
{\bf\Large On recovering non-local perturbation of non-selfadjoint Sturm--Liouville operator}
\end{center}

\begin{center}
{\bf\large Maria Kuznetsova\footnote{Department of Mathematics, Saratov State University, e-mail: {\it kuznetsovama@info.sgu.ru}}}
\end{center}

\noindent{\bf Abstract.} 
Recently, there appeared a significant interest in inverse spectral problems for non-local operators arising in numerous applications. In the present work, we consider the operator with frozen argument $ly = -y''(x) + p(x)y(x) + q(x)y(a),$ which is a non-local perturbation of the non-selfadjoint Sturm--Liouville operator. We study the inverse problem of recovering the potential $q\in L_2(0, \pi)$ by the spectrum when the coefficient $p\in L_2(0, \pi)$ is known.  While the previous works were focused only on the case $p=0,$ here we investigate the more difficult non-selfadjoint case, which requires consideration of eigenvalues multiplicities. We develop an approach based on the relation between the characteristic function and the coefficients $\{ \xi_n\}_{n \ge 1}$ of the potential $q$ by a certain basis. We obtain necessary and sufficient conditions on the spectrum being asymptotic formulae of a special form. They yield that a part of the spectrum does not depend on $q,$ i.e. it is uninformative.
For the unique solvability of the inverse problem, one should supplement the spectrum with a part of the coefficients $ \xi_n,$ being the minimal additional data. For the inverse problem by the spectrum and the additional data, we obtain a uniqueness theorem and an algorithm.

\medskip
\noindent {\it Key words}: inverse spectral problems, frozen argument, Sturm--Liouville operators, non-local operators, necessary and sufficient conditions

\medskip
\noindent {\it 2010 Mathematics Subject Classification}: 34K29, 34A55
\\

\section*{Introduction}
Inverse spectral problems consist in recovering operators from their spectral characteristics. The classical results in this field were obtained for the  differential operators~\cite{KuM_B,KuM_yurko,KuM_march,KuM_levitan,KuM_high-order}, which are local.
Recently, in connection with numerous applications, there appeared a considerable interest in inverse problems for non-local operators~\cite{KuM_djuric, KuM_pikula2, KuM_yur-con,  KuM_uniform,KuM_results2007,KuM_bolletin, KuM_nonlocal mc}, and, in particular, for the operators with frozen argument~\cite{KuM_albeverio, KuM_nizh10,KuM_BBV, KuM_BV, KuM_BK, KuM_wang, KuM_but-hu,
KuM_tsai, KuM_my-froz, KuM_bon-disc,KuM_my-uniform,KuM_DobHry,KuM_AML}. 
Their studying is complicated by the fact that non-local operators require development of non-standard methods.

In this paper, we study the recovery of a complex-valued potential $q \in L_2(0, \pi)$ by the spectrum $\{ \lambda_n \}_{n \ge 1}$ of the boundary value problem
\begin{equation} \label{KuM_frozen S-L}
ly := -y''(x) + p(x) y(x) + y(a) q(x) = \lambda y(x), \quad x \in (0, \pi),
\end{equation}
\begin{equation} \label{KuM_bc}
y^{(\alpha)}(0) = y^{(\beta)}(\pi) = 0, \quad \alpha, \beta \in \{ 0, 1\},
\end{equation}
where $a \in [0, \pi]$ and  $p \in L_2(0, \pi)$ is complex-valued. Operators $l y$
are usually called Sturm--Liouville operators with {\it frozen argument}. 
They have a close relation to the operators with integral boundary conditions~\cite{KuM_lomov,KuM_lomov2,KuM_kraal, KuM_bolletin, KuM_nonlocal mc}, which arise in studying diffusion and
heating processes and in the theory of elasticity, see~\cite{KuM_gord, KuM_feller, KuM_feller2, KuM_katarzyna}. 
In connection with this, in~\cite{KuM_pol}, some spectral properties of the operator $ly$ were established in the case of periodic boundary conditions. However, the mentioned work does not address an inverse spectral problem.

The previous studies of inverse spectral problems for Sturm--Liouville operators with frozen argument~\cite{KuM_albeverio, KuM_nizh10,KuM_BBV, KuM_BV, KuM_BK, KuM_wang, KuM_but-hu,
KuM_tsai, KuM_my-froz, KuM_bon-disc,KuM_my-uniform,KuM_DobHry,KuM_AML} were focused only on the case $p=0.$ 
In this particular case, a comprehensive study of recovering $q$ by the spectrum required a series of the works~\cite{KuM_BBV, KuM_BV, KuM_BK, KuM_wang,KuM_tsai, KuM_my-uniform, KuM_AML}.
The most general approach to the operator with frozen argument was developed in~\cite{KuM_AML}, which allowed us to obtain necessary and sufficient conditions on the spectrum and, afterwards, a uniform stability of the inverse problem~\cite{KuM_my-uniform}.

 In~\cite{KuM_DobHry}, there was suggested another approach to the operator with frozen argument within the framework of perturbation theory.
 According to it, operator $l y$ could be treated as a one-dimensional perturbation of the Sturm--Liouville operator $Ay := -y''+p(x) y.$
However, for studying the spectral properties of $l y,$ this approach needs the selfadjointness of the operator $A.$ By this reason, it  is inapplicable to the case of complex-valued $p$ considered here, since the unperturbed operator $A$ is non-selfadjoint.

Here, we develop an approach to the general situation of arbitrary $p \in L_2(0, \pi),$ relying on some ideas of~\cite{KuM_AML}. 
We significantly extend the mentioned ideas to take into account the non-selfadjointness of the unperturbed operator, which requires consideration of eigenvalues multiplicities.

Main results of the paper consist in necessary and sufficient conditions on the spectrum, a uniqueness theorem and an algorithm.
As in the previous work~\cite{KuM_AML}, the necessary and sufficient conditions are asymptotic formulae of a special form. They give the so-called degeneration condition that some part of the eigenvalues is uninformative, i.e. it does not depend on $q.$ 
 However, compared to~\cite{KuM_AML}, the non-selfadjointness of the unperturbed operator leads to a new effect that the degeneration condition includes a restriction on the minimal possible multiplicity of each uninformative eigenvalue.

The paper is organized as follows. In Section~1, we introduce necessary objects and provide auxiliary statements. 
As well we obtain a characteristic function of boundary value problem~\eqref{KuM_frozen S-L},~\eqref{KuM_bc} and a main equation of the inverse problem. 
In Section~2, by necessity, we establish the conditions on the spectrum. The main results and their proves are given in Section~3. 

\section{Preliminaries. Main equation of the inverse problem}
%\vspace*{5mm}

First, we consider the unperturbed boundary value problem for the classical Sturm--Liouville equation 
\begin{equation} \label{KuM_S-L}
-y''(x) + p(x) y(x) = \lambda y(x), \quad x \in (0, \pi),
\end{equation}
with boundary conditions~\eqref{KuM_bc}.
Denote by $S_a(x, \lambda)$ and $C_a(x, \lambda)$ the solutions of equation~\eqref{KuM_S-L} under the initial conditions 
$$C_a(a, \lambda) = 1, \; C'_a(a, \lambda) =0, \quad S_a(a, \lambda) = 0, \; S'_a(a, \lambda) =1;$$ 
here and below the prime symbol stands for the derivative with respect to the { \it first} argument.
Let us agree that $\int_a^b f(t) \, dt$ is understood as $-\int_b^a f(t) \, dt$ when $b < a.$
For $x \in [0, \pi],$ using transformation operators (see, e.g.~\cite{KuM_yurko}), we obtain the following representations:
\begin{equation} \label{KuM_transformation}
\left.\begin{array}{c}
\displaystyle S_a(x, \lambda) = \frac{\sin \rho(x-a)}{\rho} - \frac{\cos \rho(x-a)}{\rho^2} \omega_a(x) + \int_a^x K_{a,1}(x, t) \frac{\cos \rho(t-a)}{\rho^2} \, dt, \\[4mm]
\displaystyle S'_a(x, \lambda) = \cos \rho(x-a) + \frac{\sin \rho(x-a)}{\rho} \omega_a(x) + \int_a^x K_{a,2}(x, t) \frac{\sin \rho(t-a)}{\rho} \, dt, \\[4mm]
\displaystyle C_a(x, \lambda) = \cos \rho(x-a) + \frac{\sin \rho(x-a)}{\rho} \omega_a(x) + \int_a^x K_{a,3}(x, t) \frac{\sin \rho(t-a)}{\rho} \, dt, \\[4mm]
\displaystyle C'_a(x, \lambda) = -\rho \sin \rho(x-a) + \cos \rho(x-a) \omega_a(x) + \int_a^x K_{a,4}(x, t) \cos \rho(t-a) \, dt,
\end{array}
\right\}
\end{equation}
where $\rho^2=\lambda$ and $\omega_a(x) = \frac12 \int_a^x p(t)\,dt.$ 
In~\eqref{KuM_transformation}, for each fixed $x,$ there determined $K_{a,j}(x,  \cdot) \in L_2(0, \pi),$  and $K_{a,j} \in L_2[0, \pi]^2,$ $j=\overline{1, 4}.$

Consider the entire function $\Delta_0(\lambda) = C_a^{(\alpha)}(0,\lambda)S_a^{(\beta)}(\pi, \lambda) - S_a^{(\alpha)}(0,\lambda)C_a^{(\beta)}(\pi, \lambda),$ wherein $y^{(0)} := y$ and $y^{(1)} := y'.$ 
It is easy to see that $\Delta_0$ is a characteristic function of unperturbed boundary value problem~\eqref{KuM_bc},~\eqref{KuM_S-L}, i.e. a number $\mu_n$ is its eigenvalue if and only if $\Delta_0(\mu_n) =0.$
By $\{ \mu_n \}_{n \ge 1}$ we denote the spectrum of~\eqref{KuM_bc},~\eqref{KuM_S-L},
 being the sequence of the eigenvalues taken with the account of algebraic multiplicities. The following asymptotics are known (see~\cite{KuM_yurko}):
\begin{equation} \label{KuM_mu}
\mu_n = \theta_n^2, \quad \theta_n =n -\frac{\alpha+\beta}{2} + \frac{\omega}{\pi n} +\frac{\kappa_n}{n}, \; n \ge 1, \quad \omega :=\frac{1}{2} \int_0^\pi p(t) \, dt, \quad \{ \kappa_n \}_{n\ge1} \in \ell_2.
\end{equation}

By $m_n$ we denote multiplicity of the eigenvalue $\mu_n.$ 
By asymptotics~\eqref{KuM_mu}, for a sufficiently large $n,$ we have $m_n=1.$ 
Without loss of generality, we assume that equal values in the spectrum follow each other.  
Then,  we have
$$\mu_n = \mu_{n+1} = \ldots = \mu_{n+m_n -1}, \quad n \in \mathcal S, \quad {\mathcal S} := \{ n \ge 2 \colon \mu_n \ne \mu_{n-1}\}\cup\{1\}.$$
The index $n \in \mathcal S$ corresponds to the unique elements in $\{ \mu_n \}_{n \ge 1},$ while for $n \in \mathcal S$ and $\nu \in \overline{0, m_n - 1},$ the index $k = n+\nu$ runs through $\mathbb N.$ 
%%%%%%%%%%%%%%%%%%%%%%%%%%%%%%%%%%%%

Now, we are ready to study boundary value problem~\eqref{KuM_frozen S-L},~\eqref{KuM_bc}.
Introduce the solutions $S(x, \lambda)$ and $C(x, \lambda)$ of equation~\eqref{KuM_frozen S-L} under the initial conditions
\begin{equation*} \label{KuM_initial conditions}
S(a, \lambda) =  0, \; S'(a, \lambda) =1, \quad
 C(a, \lambda) = 1, \; C'(a, \lambda) = 0.\end{equation*}
Any other solution of~\eqref{KuM_frozen S-L} is a linear combination of $S(x, \lambda)$ and $C(x, \lambda).$
It is easy to see that 
\begin{equation} \label{KuM_SC}
S(x, \lambda) = S_a(x, \lambda),  \quad  C(x, \lambda) = C_a(x,\lambda) + \int_a^x W(x, t, \lambda) q(t) \, dt,\end{equation}
where $W(x, t, \lambda) := C_a(t, \lambda) S_a(x, \lambda) - C_a(x, \lambda) S_a(t, \lambda).$ 
We introduce the entire function
\begin{equation} \label{KuM_Delta def}
\Delta(\lambda) = C^{(\alpha)}(0, \lambda) S^{(\beta)}(\pi, \lambda) - S^{(\alpha)}(0, \lambda) C^{(\beta)}(\pi, \lambda).\end{equation}
Then, $\Delta(\lambda)$ is a characteristic function of boundary value problem~\eqref{KuM_frozen S-L},~\eqref{KuM_bc}, while $\{ \lambda_n \}_{n \ge 1}$ is a sequence of its zeroes taken with the account of multiplicities.

Substituting~\eqref{KuM_SC} into~\eqref{KuM_Delta def}, we obtain
\begin{equation} \label{KuM_Delta rep}
\Delta(\lambda) = \Delta_0(\lambda) - S_a^{(\beta)}(\pi, \lambda) \int_0^a W^{(\alpha)}(0, t, \lambda) \, q(t)\,dt
-S_a^{(\alpha)}(0, \lambda) \int_a^\pi W^{(\beta)}(\pi, t, \lambda) \,q(t)\,dt.
\end{equation}
Following the approach in~\cite{KuM_AML}, we should substitute into~\eqref{KuM_Delta rep} the values $\lambda=\mu_n,$ being the zeroes of the main part $\Delta_0(\lambda).$
 In the paper~\cite{KuM_AML} corresponding to the case $p=0,$ we had $\mu_n = \big(n - \frac{\alpha+\beta}{2}\big)^2,$ being simple eigenvalues. Here, we have to take into account that $\mu_n$ may be multiple. 
For each $n \in \mathcal S,$ we differentiate the both parts of formula~\eqref{KuM_Delta rep} $\nu = \overline{0, m_n -1}$ times and put $\lambda=\mu_n.$ Since 
\begin{equation} \label{KuM_det rel}
\frac{\partial^\nu}{\partial \lambda^\nu}\big[C^{(\alpha)}_a(0, \lambda)S^{(\beta)}_a(\pi, \lambda)\big] = \frac{\partial^\nu}{\partial \lambda^\nu}[S^{(\alpha)}_a(0, \lambda)C^{(\beta)}_a(\pi, \lambda)],  \quad   n \in {\cal S}, \;\nu =\overline{0, m_n-1},
\end{equation}
in $\lambda = \mu_n,$ we obtain
\begin{equation} \label{KuM_aux}
\Delta^{(\nu)}(\lambda) = \left[S_a^{(\beta)}(\pi, \lambda) \int_0^\pi g(t, \lambda) q(t)\, dt\right]^{(\nu)}, \quad \lambda = \mu_n, \; n \in {\cal S}, \; \nu =\overline{0, m_n-1},
\end{equation}
where $g(t, \lambda) := -W^{(\alpha)}(0,t,\lambda).$ One can see that $g(t, \lambda) = S_0(x, \lambda)$ if $\alpha = 0,$ and 
 $g(t, \lambda) = -C_0(x, \lambda)$ if $\alpha = 1.$ Moreover, $g^{(\beta)}(\pi, \lambda) = \Delta_0(\lambda).$

For $n \in {\cal S}$ and $\nu =\overline{0, m_n-1},$ we introduce
$$a_{n+\nu} = \frac{n^{1-\beta}}{\nu!}\left. \frac{\partial^\nu S_a^{(\beta)}(\pi, \lambda)}{\partial\lambda^\nu}\right|_{\lambda=\mu_n},\quad g_{n+\nu}(t) = \frac{n^{1-\alpha}}{(m_n-\nu-1)!} \left.\frac{\partial^{m_n-\nu-1} g(t, \lambda)}{\partial\lambda^{m_n-\nu-1}} \right|_{\lambda=\mu_n}. $$
Note that $a_n$ and $g_n$ are the objects constructed by unperturbed boundary value problem~\eqref{KuM_bc},~\eqref{KuM_S-L}, and they are known.
Formulae~\eqref{KuM_transformation} and~\eqref{KuM_mu} yield that
$a_n=O(1)$ and $\{ g_n(t) \}_{n \ge 1}$ is an almost normalized system.
At the same time, $\{ g_n(t) \}_{n \ge 1}$ is constructed from eigen- and associated functions of the operator $Ay = -y'' + p(x) y$ considered under strongly regular conditions~\eqref{KuM_bc}. By the mentioned reasons, the following proposition holds (see, e.g.,~\cite{KuM_shka, KuM_nai}).

\begin{Proposition} \label{KuM_prop}
The functional sequence $\{g_n(t)\}_{n \ge 1}$ is a Riesz basis in $L_2(0, \pi).$
\end{Proposition}

Applying the general Leibniz rule in~\eqref{KuM_aux}, we get
\begin{equation} \label{KuM_main equation}
n^{2 - \alpha-\beta}\frac{\Delta^{(\nu)}(\mu_n)}{\nu!} = \sum_{\eta=0}^\nu a_{n+\nu-\eta} \xi_{n+m_n-1-\eta}, \quad n \in {\cal S}, \; \nu =\overline{0, m_n-1},
\end{equation}
where $\xi_{k} := \int_0^\pi g_{k}(t) q(t) \, dt.$ 
Further, we consider~\eqref{KuM_main equation} as an equation with respect to the coefficients $\{ \xi_n\}_{n \ge 1}.$ From Proposition~\ref{KuM_prop} it follows that $\{ \xi_n\}_{n\ge 1} \in \ell_2$ and its knowledge allows one to recover $q$ uniquely. Formula~\eqref{KuM_main equation} is called {\it main equation} of the inverse problem.

%%%%%%%%%%%%%%%%%%%%%%%%%%%%%%%%
\section{Necessary conditions}

In this section, we obtain necessary conditions on the spectrum $\{ \lambda_n \}_{n\ge 1},$ which consist in asymptotic formulae~\eqref{KuM_main asym}. In the next section, we prove that they are the sufficient conditions as well.

Remind that $\lambda_n=\rho_n^2$ and $\mu_n = \theta_n^2.$ Without loss of generality, we assume that $\arg\rho_n, \, \arg \theta_n \in [-\frac\pi2, \frac\pi2).$ We start by obtaining the weakest asymptotics for $\rho_n.$

\begin{Lemma} \label{KuM_L1}The following asymptotics hold:
$\rho_n = \theta_n + o(1),$ $n \ge 1.$
\end{Lemma}

\begin{Proof}
For definiteness, we provide computations in the case $\alpha=1$ and $\beta=0$ (the other cases are proceeded analogously).
By asymptotics~\eqref{KuM_mu}, it is sufficient to prove that $\rho_n = n -\frac12 + o(1).$ For $\lambda=\rho^2,$ we consider $\Delta(\lambda)$ as an entire function of $\rho.$ Then, $\{ \rho_n \}_{n \ge 1} \cup \{ -\rho_n \}_{n \ge 1}$ is the set of all its zeroes.
 
Denote $\tau = |\mathrm{Im}\, \rho|.$ Using the corresponding formulae in~\eqref{KuM_transformation}, for $\rho \to \infty,$ we obtain
\begin{equation*} W'(0, t, \lambda) = -C_0(t, \lambda)  = O\left( e^{\tau t}\right), \quad W(\pi, t, \lambda)  =-S_\pi(t, \lambda) =  O\left( \frac{e^{\tau (\pi-t)}}{\rho}\right), 
\label{KuM_WA}\end{equation*}
$$S_a(\pi, \lambda) = O\left(\frac{e^{\tau(\pi-a)}}{\rho}\right),\ S_a'(0, \lambda) = O\left( e^{\tau a}\right),\ \Delta_0(\lambda) = -C_0(\pi, \lambda) = -\cos \rho \pi + O\left( \frac{e^{\tau \pi}}{\rho}\right).$$
Substituting these relations into~\eqref{KuM_Delta rep}, we get
\begin{equation} \label{KuM_Delta asymp} 
\Delta(\lambda) 
=\Delta_0(\lambda) + O\left(\frac{e^{\tau \pi}}{\rho}\right)=\cos \rho \pi + O\left(\frac{e^{\tau \pi}}{\rho}\right).
\end{equation}
For any $\delta \in (0, \frac14),$ we have the following estimate with $ C_\delta>0$ (see~\cite{KuM_yurko}):
\begin{equation} \label{KuM_cos est}
|\cos \rho \pi| > C_\delta \, e^{\tau \pi}, \quad \rho \in G_\delta := \left\{ z \in \mathbb C \colon \left|z - n + \frac12\right|\ge\delta, \; n \in {\mathbb Z}\right\}.
\end{equation}
Then, by~\eqref{KuM_Delta asymp}, there exists $N_\delta \in \mathbb N$ such that 
$|\cos \rho \pi| \ge |\cos \rho \pi - \Delta(\rho^2)|$ as soon as $\rho \in G_\delta$ and $|\rho| \ge N_\delta.$
Applying Rouche's theorem, we arrive at that in the circle $|\rho| < N_\delta,$ the functions $\cos \rho \pi$ and $\Delta(\rho^2)$ have the same number of zeroes.
Analogously, in each circle $|\rho - n + \frac12|<\delta,$ where $n \in \mathbb Z$ is such that $|n - \frac12| > N_\delta,$ there is exactly one zero of $\Delta(\rho^2).$ Taking into account that $\Delta(\rho^2)$ and $\cos \rho \pi$ are even functions of $\rho,$ and that $\delta$ can be arbitrarily small, we arrive at the needed asymptotics.
\end{Proof}

Further, we clarify the obtained in Lemma~\ref{KuM_L1} necessary conditions on the spectrum.

\begin{Theorem} \label{KuM_T1} I. Let $K \in \mathbb N$ be such that for $n \ge K,$ all $m_n = 1.$  
Then, the following asymptotics hold:
\begin{equation} \label{KuM_eigenvalues asym}
\lambda_n =  \mu_n + a_n \varkappa_n, \; n \ge K, \quad \{ \varkappa_n \}_{n \ge K} \in \ell_2.
\end{equation}
II. For $n \in \mathcal S,$ denote $k_n = \min(m_n, r_n),$ where $r_n$ is  multiplicity of $\lambda=\mu_n$ as a zero of the entire function $S_a^{(\beta)}(\pi, \lambda).$
Then, there exists such numeration of $\{ \lambda_n \}_{n \ge 1}$ that
\begin{equation} \label{KuM_chain}
\lambda_n = \lambda_{n+1} = \ldots = \lambda_{n+k_n-1} = \mu_n, \quad n \in {\mathcal S}.
\end{equation}
\end{Theorem}
\begin{Proof}

I. By Lemma~\ref{KuM_L1}, we have $\rho_n = \theta_n + \eta_n,$ where $\eta_n = o(1).$ 
To prove~\eqref{KuM_eigenvalues asym}, we substitute $\lambda=\rho_n^2$ into~\eqref{KuM_Delta rep} and, using Taylor series, obtain asymptotics for $\eta_n,$ where $n \ge K.$ 

For definiteness, consider the case $\alpha=1$ and $\beta=0$ (the other cases are proceeded analogously).
 Since $\Delta_0(\lambda) = -C_0(\pi, \lambda),$ substituting $\lambda=\rho_n^2$ into the corresponding formula in~\eqref{KuM_transformation},  applying trigonometric formulae along with asymptotics~\eqref{KuM_mu} and $\eta_n = o(1),$ we get
$$
\Delta_0(\lambda_n) = -\cos \theta_n\pi - \frac{\omega}{\theta_n} \sin \theta_n\pi - \int_0^\pi K_{0,3}(\pi, t) \frac{\sin \theta_nt}{\theta_n} \, dt + (-1)^{n+1} \eta_n \pi  + o(\eta_n).
$$
In this formula, the first three summands compose $\Delta_0(\mu_n) = 0,$ and we arrive at 
\begin{equation} \label{KuM_manip1}
\Delta_0(\lambda_n) = \eta_n\big((-1)^{n+1} + o(1)\big).\end{equation}

Proceeding analogously, based on~\eqref{KuM_transformation}, we also obtain the following asymptotics:
\begin{equation} \left.
 \begin{array}{cc} \label{KuM_SC lambda^0}
S_a(t, \lambda_n) = S_a(t, \mu_n) + O(n^{-1}\eta_n), & S_a'(t, \lambda_n) = S_a'(t, \mu_n) + O(\eta_n), \\[4mm]
C_a(t, \lambda_n) = C_a(t, \mu_n) + O(\eta_n), & C_a'(t, \lambda_n) = C_a'(t, \mu_n) + O(n\eta_n)
\end{array} \right\}
\end{equation}
uniformly on $t \in [0, \pi].$
Using~\eqref{KuM_SC lambda^0}, we have
\begin{multline*}
A_n:=S_a(\pi, \lambda_n) \int_0^a W'(0, t,  \lambda_n) q(t)\,dt = \\
= \left(S_a(\pi, \mu_n) + O\left(\frac{\eta_n}{n}\right)\right)\left(-\int_0^a g_n(t) q(t) \, dt + O(\eta_n)\right).
\end{multline*}
Since $g_n(t) = -\cos \theta_n t + O(n^{-1}),$ by the Riemann--Lebesgue lemma, $\int_0^a g_n(t) q(t) \, dt = o(1).$
Using also that $S_a(\pi, \mu_n) = O(n^{-1}),$ we obtain
\begin{equation} \label{KuM_manip2}
A_n = -S_a(\pi, \mu_n) \int_0^a g_n(t) q(t) \, dt + o(\eta_n).
\end{equation}

Analogously, we have
$$
B_n := S_a'(0, \lambda_n) \int_a^\pi W(\pi, t,  \lambda_n)q(t)\,dt = 
S_a'(0, \mu_n) \int_a^\pi W(\pi, t,  \mu_n)q(t)\,dt + o(\eta_n).
$$
Relation~\eqref{KuM_det rel} yields that 
$
S_a'(0, \mu_n) W(\pi, t,  \mu_n) =  -S_a(\pi, \mu_n) g_n(t),
$
and
\begin{equation} \label{KuM_manip3}
B_n = -S_a(\pi, \mu_n) \int_a^\pi g_n(t) q(t) \, dt + o(\eta_n).
\end{equation}
Combining~\eqref{KuM_manip1},~\eqref{KuM_manip2}, and~\eqref{KuM_manip3} with~\eqref{KuM_Delta rep} in $\lambda=\rho_n^2,$ we obtain
$$0 =   ((-1)^n + o(1)) \eta_n+S(\pi, \mu_n) \xi_n, \quad \{ \xi_n \}_{n \ge 1} \in \ell_2.$$
For $n \ge K,$ the value $\eta_n$ is the unique solution of this equation,
which leads to $\eta_n = a_n \nu_n n^{-1}$ with $\{ \nu_n\}_{n \ge K} \in \ell_2,$  and to~\eqref{KuM_eigenvalues asym}.

II. From the definition it follows that  
\begin{equation} \label{KuM_chain a}
a_n = a_{n+1} = \ldots = a_{n+k_n-1} = 0.
\end{equation} Then, by formula~\eqref{KuM_main equation},  $\mu_n$ is a zero of $\Delta(\lambda)$ of multiplicity not less then $k_n.$ This means that $\mu_n$ occurs in the spectrum $\{ \lambda_n \}_{n \ge 1}$ at least $k_n$ times, and~\eqref{KuM_chain} holds up to a numeration. 
\end{Proof}
%%%%%%%%%%%%%%%%%%%%%%%%%%

In what follows, we can assume that the numeration of $\{ \lambda_n\}_{n\ge 1}$ satisfies~\eqref{KuM_chain}.
Denote \begin{equation} \label{KuM_Omega}
\Omega = \{ n + \nu \colon n \in {\mathcal S}, \; \nu = \overline{0, k_n - 1}\}, \quad \overline{\Omega} = {\mathbb N} \setminus \Omega. \end{equation}
Formula~\eqref{KuM_chain} yields that the part of the spectrum $\{ \lambda_n \}_{n \in \Omega}$ does not depend on $q,$ i.e. we have the degeneration condition. Each unique eigenvalue $\lambda_n = \mu_n$ in this part occurs $k_n$ times, which restricts its multiplicity to be not less than $k_n.$ 
Note that~\eqref{KuM_chain} follows from~\eqref{KuM_eigenvalues asym} when  $n \ge K$ since $k_n \le 1.$ In~\cite{KuM_AML}, condition~\eqref{KuM_chain}  was not required because for $p=0,$ we can take $K=1.$

Now, we unify conditions~\eqref{KuM_eigenvalues asym} and~\eqref{KuM_chain} into one  formula. Introduce the values 
$$b_{n+\nu} = \left\{ \begin{array}{cc}
a_{n+\nu}, & \nu = \overline{0, p_n - 1}, \\[2mm]
1, & \nu = \overline{p_n, m_n-1},
\end{array}
\right. \; p_n := \max(1, k_n), \quad n \in {\mathcal S}.$$
Then, formulae~\eqref{KuM_eigenvalues asym} and~\eqref{KuM_chain} are particular cases of the following relation:
\begin{equation} \label{KuM_main asym}
\lambda_n =  \mu_n + b_n \varkappa_n, \quad n \ge 1, \quad \{ \varkappa_n\}_{n \ge 1} \in \ell_2.
\end{equation}
 %This does not depend on N
 It differs from ~\eqref{KuM_eigenvalues asym} and~\eqref{KuM_chain} only by a finite number of formulae for $n=k+\nu < K,$ $k \in \mathcal S,$ $\nu \in \overline{k_n, m_n-1}$ with $b_n \ne 0,$ being non-restrictive. By this reason,~\eqref{KuM_main asym} is equivalent to~\eqref{KuM_eigenvalues asym} along with~\eqref{KuM_chain}. In particular, $n \in \Omega$ if and only if $b_n = 0.$
 
Note that $b_n=O(1),$ then the following asymptotics is weaker than~\eqref{KuM_main asym}: 
\begin{equation} \label{KuM_sing asym}
\lambda_n=\rho_n^2, \quad \rho_n = n - \frac{\alpha+\beta}{2} + \frac{\omega}{\pi n} + \frac{\nu_n}{n}, \quad \{ \nu_n\}_{n\ge 1} \in \ell_2.
\end{equation}
 By the standard approach involving Hadamard's factorization theorem (see, e.g.,~\cite{KuM_yurko}), one can prove that
\begin{equation} \label{KuM_prod}
\Delta(\lambda) = (-1)^\alpha \pi^{\delta_{\alpha,\beta}}\prod_{k=1}^\infty \frac{\lambda_k - \lambda}{\xi_k}, \quad \xi_k := \left\{ \begin{array}{cc}
\left(k - \frac{\alpha+\beta}{2}\right)^2, & k \ge 1 \text{ or } \alpha+\beta <2, \\[2mm]
1, & k = 1, \; \alpha=\beta=1.
\end{array}
\right.
\end{equation}
Thus, by the spectrum we can uniquely reconstruct the characteristic function $\Delta(\lambda).$ 

\section{Main results}

First, we obtain the necessary and sufficient conditions on the spectrum.
\begin{Theorem} \label{KuM_T2}For an arbitrary sequence $\{ \lambda_n \}_{n \ge 1}$ of complex numbers to be the spectrum of boundary value problem~\eqref{KuM_frozen S-L},~\eqref{KuM_bc} with some $q \in L_2(0, \pi),$ it is necessary and sufficient to satisfy~\eqref{KuM_main asym}.
\end{Theorem}

For the proof, we need the following lemma.

\begin{Lemma} \label{KuM_L2}
Let $\Delta(\lambda)$ be constructed via~\eqref{KuM_prod},
where arbitrary values $\{ \lambda_{n}\}_{n\ge1}$ satisfy asymptotics~\eqref{KuM_sing asym}.
 Then, the following representation holds:
 \begin{equation}\label{KuM_rD0}
\Delta(\lambda) = \left\{ \begin{array}{cc}
 \displaystyle\rho^{2\alpha}\left(\frac{\sin \rho \pi}{\rho} - \frac{\cos \rho \pi}{\rho^2} \omega+ \int_0^\pi \frac{\cos \rho t}{\rho^2} \,W(t)\, dt\right), & \alpha=\beta, \\[4mm]
 \displaystyle(-1)^\alpha\left( \cos \rho \pi +\frac{\sin \rho\pi}{\rho}\omega+ \int_0^\pi \frac{\sin \rho t}{\rho}\, W(t)\, dt\right), & \alpha\ne \beta,
\end{array}
\right.
\end{equation}
where $W \in L_2(0, \pi).$ \end{Lemma}

For $\alpha=\beta=0,$ the statement of the lemma easily follows from Lemma~3.3 in~\cite{KuM_results2007} after integration in parts. For the other combinations of $\alpha$ and $\beta,$ the needed statements are proved by analogous computations. 

\begin{proof}[Proof of Theorem~\ref{KuM_T2}] The necessity part was proved in the previous section. Let us prove the sufficiency part. 
Construct the function $\Delta(\lambda)$ via formula~\eqref{KuM_prod} using the given numbers  $\{ \lambda_n \}_{n \ge 1}.$ Condition~\eqref{KuM_main asym} yields asymptotics~\eqref{KuM_sing asym},  and, by Lemma~\ref{KuM_L2}, $\Delta(\lambda)$ has the form~\eqref{KuM_rD0}.

Now, we should find a function $q$ such that its coefficients $\xi_k = \int_0^\pi q(t) g_k(t)\, dt$ satisfy~\eqref{KuM_main equation}. For every $n \in \mathcal S,$ relation~\eqref{KuM_main equation} can be considered as a system of $m_n$ linear equations with respect to the vector $[ \xi_{n +\nu} ]_{\nu=0}^{m_n-1}:$
\begin{equation} \label{KuM_system}
\left\{\begin{array}{c}
a_n \xi_{n + m_n - 1} = n^{2 - \alpha-\beta}\Delta(\mu_n), \\[3mm]
\displaystyle a_n \xi_{n+m_n-2} + a_{n+1} \xi_{n + m_n - 1} = n^{2 - \alpha-\beta}\frac{\Delta'(\mu_n)}{1!}, \\[3mm]
\ldots \\
\displaystyle a_n \xi_n + a_{n+1} \xi_{n+1} + \ldots + a_{n+m_n-1} \xi_{n+m_n-1} = 
n^{2 - \alpha-\beta}\frac{\Delta^{(m_n-1)}(\mu_n)}{(m_n-1)!}.
\end{array}\right.
\end{equation}
By~\eqref{KuM_main asym}, we have $\Delta(\mu_n) = \Delta'(\mu_n) = \ldots =  \Delta^{(k_n - 1)}(\mu_n) = 0.$ This along with~\eqref{KuM_chain} yields that first $k_n$ rows in system~\eqref{KuM_system} turn trivial identities and that arbitrary values of $\xi_{n +\nu},$ $\nu=\overline{0, k_n-1},$ satisfy this system. If $k_n < m_n,$ then $a_{n + k_n} \ne 0,$ and
 the rest $\xi_{n +\nu}$ are uniquely determined by subsequent application of the following formulae:
\begin{equation} \label{KuM_xi} 
\left.\begin{array}{c}
\displaystyle \xi_{n+m_n-1} = n^{2 - \alpha-\beta}\frac{\Delta^{(k_n)}(\mu_n)}{a_{n + k_n} k_n!}, \\[3mm] 
\displaystyle\xi_{n+m_n-\nu} = \frac{1}{a_{n + k_n}} \left( n^{2 - \alpha-\beta}\frac{\Delta^{(k_n+\nu-1)}(\mu_n)}{ (k_n+\nu-1)!} - \sum_{\eta=1}^{\nu-1} a_{n+k_n+\eta} \xi_{n+m_n-\nu +\eta}\right), \;
\nu = \overline{2, m_n-k_n}.
\end{array}\right\}
\end{equation}
Remind that for a sufficiently large $n \ge K,$ we have $m_n=1,$ and $\xi_n$ either can be arbitrary (if $k_n = 1$) or it is computed via the first formula in~\eqref{KuM_xi} (if $k_n = 0$).

 Thus, we arrive at that the part of the coefficients $\{ \xi_k\}_{k \in \overline\Omega}$ is uniquely determined by $\{ \lambda_n\}_{n\ge 1},$ while  $\{ \xi_k\}_{k \in\Omega}$ can be arbitrary (for the definition of $\Omega$ and $\overline \Omega,$ see~\eqref{KuM_Omega}).
Applying the scheme from the proof of Theorem~1 in~\cite{KuM_AML}, using representation~\eqref{KuM_rD0}, we obtain that $\{n^{2-\alpha-\beta}\Delta(\mu_n)a^{-1}_n\}_{n\in \overline\Omega} \in \ell_2,$ and $\{ \xi_k\}_{k \in \overline\Omega} \in \ell_2.$
 Choose arbitrary coefficients $\{ \xi_{k}\}_{k \in \Omega} \in l_2.$ Then, there exists $q \in L_2(0, \pi)$ such that its coefficients with respect to the basis $\{ g_k(t)\}_{k\ge1}$ are $\{ \xi_{k}\}_{k \ge 1}.$
 
 Consider boundary value problem~\eqref{KuM_frozen S-L},~\eqref{KuM_bc} with such potential $q.$ Let $\Delta_*(\lambda)$ be the characteristic function of this boundary value problem. Then, by construction, 
 $$ F(\lambda) = \frac{\Delta_*(\lambda) - \Delta(\lambda)}{\Delta_0(\lambda)}$$
 is an entire function. %Consider $F(\rho^2)$ being an entire function of $\rho.$
 Representations~\eqref{KuM_Delta rep} and~\eqref{KuM_rD0} along with~\eqref{KuM_transformation} yield asymptotics 
 \begin{equation} \label{KuM_a1}\Delta_*(\lambda) - \Delta(\lambda) = O\left(\rho^{\alpha+\beta-2} e^{\tau \pi}\right), \quad  \rho^2 = \lambda, \; \tau = |{\mathrm Im} \rho|.\end{equation}
 Using~\eqref{KuM_transformation}, we also arrive at
$$\Delta_0(\lambda) = \left\{\begin{array}{cc}
\displaystyle(-1)^\alpha\left(\cos \rho \pi + O\Big(\frac{e^{\tau \pi}}{\rho}\Big)\right), & \alpha\ne\beta, \\[3mm]
\displaystyle\rho^{2\alpha-1}\left(\sin \rho \pi + O\Big(\frac{e^{\tau \pi}}{\rho}\Big)\right), & \alpha=\beta.
\end{array} \right.
$$
Consider arbitrary $\delta \in (0, \frac14).$ For a sufficiently large $|\lambda| \ge N_\delta,$ analogously to~\eqref{KuM_cos est}, one can prove that 
\begin{equation}\label{KuM_a2}
|\Delta_0(\lambda)| \ge C_\delta |\rho|^{\alpha+\beta-1} e^{\tau \pi}, \quad \rho \in G_\delta = \left\{ z \in \mathbb C \colon \Big|z - n + \frac{\alpha+\beta}{2} \Big| \ge \delta, \; n \in \mathbb Z\right\},
\end{equation}
where $ C_\delta >0.$
 Using~\eqref{KuM_a1} and~\eqref{KuM_a2}, we arrive at $F(\lambda)=o(1)$ in $G_\delta.$ By the maximum modulus principle and Liouville's theorem, $F(\lambda) = 0.$  
Thus, the function $\Delta(\lambda)$ is the characteristic function of  the boundary value problem~\eqref{KuM_frozen S-L},~\eqref{KuM_bc} with the considered potential $q,$ and $\{ \lambda_n\}_{n \ge 1}$ is its spectrum.
\end{proof}

From the proof of Theorem~\ref{KuM_T2}, it follows that the potentials $q$ corresponding to one and the same spectrum have the same coefficients $\{ \xi_n \}_{n \in \overline \Omega},$ while for $n \in \Omega,$ the coefficients $\xi_n$ may differ. 
At the same time, by Proposition~\ref{KuM_prop}, the mapping $q \mapsto \{ \xi_n \}_{n \ge 1}$ 
 is a one-to-one correspondence between $L_2(0, \pi)$ and $\ell_2.$ Thus, for a fixed spectrum $\{ \lambda_n \}_{n \ge 1},$ one can construct the set of all iso-spectral potentials $q$ varying $\{ \xi_n \}_{n \in \Omega} \in \ell_2$ or find a unique $q$ setting additionally $\{ \xi_n \}_{n \in \Omega} \in \ell_2.$ In the latter case, we obtain a uniqueness theorem.

\begin{Theorem}
Let $\{\tilde \lambda_n\}_{n \ge 1}$ be the spectrum of boundary value problem~\eqref{KuM_frozen S-L},~\eqref{KuM_bc} with a potential $\tilde q \in L_2(0, \pi),$ while $\tilde \xi_n = \int_0^\pi \tilde q(t) g_n(t) \, dt,$ $n \in \Omega.$ If $\{\lambda_n\}_{n \ge 1}=\{\tilde \lambda_n\}_{n \ge 1}$ and $\xi_n = \tilde \xi_n$ for $n \in \Omega,$ then $q = \tilde q.$
\end{Theorem}

Since the proof of Theorem~\ref{KuM_T2} is constructive, we have the following algorithm
for recovering $q$ given $\{ \lambda_n\}_{n \ge 1}$ and $\{\xi_n\}_{n \in \Omega}.$ 

\bigskip

{\bf Algorithm~1.} To recover the potential $q,$ one should:

1. Construct $\Delta(\lambda)$ via formula~\eqref{KuM_prod}.

2. For $n \in \mathcal S,$ by formula~\eqref{KuM_xi}, compute the unknown coefficients $\xi_{n+k_n}, \ldots, \xi_{n+m_n -1}.$

3. Find $q = \sum_{n=1}^\infty \xi_n f_n,$ where $\{ f_n \}_{n \ge 1}$ is the basis  biorthonormal to $\{ \overline g_n\}_{n\ge 1}$ in $L_2(0, \pi).$
\bigskip

\noindent{\it Acknowledgements.} This research was supported by grant No.~22-21-00509 of the Russian Science Foundation,
https://rscf.ru/project/22-21-00509/.

\end{document}